\documentclass[a4paper,12pt]{amsart}
\usepackage{geometry}
\usepackage{etex}
\usepackage{fullpage}
\usepackage{txfonts}
\usepackage{latexsym}
\usepackage{amssymb}
\usepackage{stmaryrd}
\usepackage{mathrsfs}
\usepackage{euscript}
\usepackage[all]{xy}
\usepackage[sans]{dsfont}
\usepackage{scalefnt}
\usepackage{enumerate}
\usepackage{mathtools}
\let\Sect\S 

\def\limp{\Rightarrow}      
\def\liff{\Leftrightarrow}  

\def\Lam#1#2{\lambda#1\,{.}\,#2}   
\def\cc{\texttt{c}\hspace{-2pt}\texttt{c}} 
\def\Fork{{\pitchfork}}            
\def\<{\langle}                    
\def\>{\rangle}                    
\def\redd{\twoheadrightarrow}      

\def\A{\mathscr{A}}      
\def\S{\mathscr{S}}      
\def\ile{\preccurlyeq}   
\def\ige{\succcurlyeq}   
\def\imeet{\curlywedge}  
\def\bigimeet{\bigcurlywedge} 
\def\iand{\sqcap}        
\def\ior{\sqcup}         
\def\ilambda{\boldsymbol{\lambda}}  
\def\ent{\vdash}         

\newcommand\ScaleForall[1]{\vcenter{\hbox{\scalefont{#1}$\forall$}}}
\newcommand\ScaleExists[1]{\vcenter{\hbox{\scalefont{#1}$\exists$}}}
\DeclareMathOperator*\bigforall{%
  \vphantom\sum\mathchoice{\ScaleForall{2}}{\ScaleForall{1.4}}{\ScaleForall{1}}{\ScaleForall{0.75}}}
\DeclareMathOperator*\bigexists{%
  \vphantom\sum\mathchoice{\ScaleExists{2}}{\ScaleExists{1.4}}{\ScaleExists{1}}{\ScaleExists{0.75}}}

\def\Hom{\mathrm{Hom}}         
\def\id{\mathrm{id}}           
\def\op{\mathrm{op}}           
\def\Set{\mathbf{Set}}         
\def\HA{\mathbf{HA}}           
\def\Asm{\mathbf{Asm}}
\def\AsmA{\Asm_\A}            
\def\Car#1{\underline{#1}}     
\def\Ex#1{\mathbf{E}_{#1}}  
\def\Im{\mathrm{Im}}           
\def\Ker{\mathrm{Eq}}         
\def\Coker{\mathrm{Coeq}}     
\def\NN{\mathbf{N}}            
\def\P{\mathsf{P}}             
\def\Anon{\text{--}}           
\def\N{\mathds{N}}             
\def\Pow{\mathfrak{P}}         
\def\ds{\displaystyle}
\def\mcent#1{\mskip-100mu#1\mskip-100mu}
\long\def\COUIC#1{}

\def\Im{\mathrm{Im}}           
\def\SOC{\mathbf{\Omega}}      

\newtheorem{theorem}{Theorem}[section]
\newtheorem{proposition}[theorem]{Proposition}
\newtheorem{fact}[theorem]{Fact}
\newtheorem{lemma}[theorem]{Lemma}

\theoremstyle{definition}
\newtheorem{definition}[theorem]{Definition}
\newtheorem{remark}[theorem]{Remark}

\newtheorem{examples}[theorem]{Examples}
\newtheorem{skol}{Scholium}
\newcounter{nnnotation}

\newtheorem{notation*}[nnnotation]{Notation}

\newcommand{\tmmathbf}[1]{\ensuremath{\boldsymbol{#1}}}
\newcommand{\tmop}[1]{\ensuremath{\operatorname{#1}}}
\newcommand{\tmrsup}[2]{\ensuremath{#1^#2}}
\newcommand{\tmrsub}[2]{\ensuremath{#1_#2}}
\newcommand{\assign}{:=}
\newcommand{\nin}{\not\in}
\newcommand{\nocomma}{}
\newcommand{\tmem}[1]{{\em #1\/}}

\newcommand{\tmtextbf}[1]{\text{{\bfseries{#1}}}}
\newenvironment{enumeratenumeric}{\begin{enumerate}[1.] }{\end{enumerate}}
\newenvironment{enumerateroman}{\begin{enumerate}[i.] }{\end{enumerate}}
\newcommand{\ens}{\ensuremath{\Set}} 
\newcommand{\arcat}[1]{\tmrsup{#1}{\rightarrow}}
\newcommand{\slice}[2]{#1/#2}
\def\arobjxy#1#2#3{
\def\objectstyle{\scriptstyle}
\def\labelstyle{\scriptscriptstyle}
\left [ \vcenter{\xymatrix @-1.3pc {#1 \ar[d]^{#2} \\ #3}} \right ]
}
\newcommand{\arobj}[3]{\arobjxy #2 #1 #3}
\def\smallarobjxy#1#2#3{
\def\objectstyle{\scriptstyle}
\def\labelstyle{\scriptscriptstyle}
 [ \hspace{-.25pc}
\vcenter  {\vspace{.085pc} \xymatrix @-1.3pc {#2: \hspace{.08pc} #1 \ar[r] & #3}}
\hspace{-.25pc} ]
}
\newcommand{\smallarobj}[3]{\smallarobjxy #2 #1 #3}
\newcommand{\pullbackcorner}[1][dr]{\save*!/#1-1.4pc/#1:(-1,1)@^{|-}\restore}
\newcommand{\depr}[1]{\tmrsub{\bigsqcap}{#1}}
\newcommand{\psect}{\ensuremath{\mathfrak{S}}}
\newcommand{\ppsect}[2]{{ #1} {\psect} { #2}}
\newcommand{\apsect}[2]{{ #1} {\psect} { #2}}
\newcommand{\lam}[2]{\Lam{#1}{#2}} 
\newcommand{\intrpa}[1]{(#1)^{\A}}
\newcommand{\exi}[1]{\Ex{#1}} 
\newcommand{\pfib}[3]{#1 \tmrsub{\times}{#2} #3}
\newcommand{\car}[1]{\Car{#1}} 
\newcommand{\lpair}[2]{\ensuremath{\tmmathbf{\left\< #1, #2 \right\>}}}
\newcommand{\lpione}{\ensuremath{\tmmathbf{\pi_1}}}
\newcommand{\lpitwo}{\ensuremath{\tmmathbf{\pi_2}}}

\def\IRIF{Institut de Recherche en Informatique Fondamentale (IRIF),
  Université Paris Cité, Bâtiment Sophie Germain, Case courrier 7014,
  8 place Aurélie Nemours, 75205 Paris cedex 13, Paris, France}
\def\IMERL{Instituto de Matem{\'a}tica y Estad{\'i}stica Rafael Laguardia
  (IMERL), Facultad de Ingenier{\'i}a, Universidad de la Rep{\'u}blica,
  Julio Herrera y Reissig 565, C.P. 11300 Montevideo, Uruguay}
\def\LAMA{Laboratoire de Math{\'e}matiques (LAMA),
  Universit{\'e} Savoie Mont Blanc,
  B{\^a}timent Le Chablais, Campus scientifique,
  73376 Le Bourget du Lac, France}

\begin{document}

\title{Implicative Assemblies}
\author[F. Castro]{Félix Castro}
\address{\IRIF{} \& \IMERL}
\author[A. Miquel]{Alexandre Miquel}
\address{\IMERL}
\author[K. Worytkiewicz]{Krzysztof Worytkiewicz}
\address{\LAMA}
\date{}
\thanks{Competing interests: The authors declare none.}
\maketitle


\begin{abstract}
    \emph{Implicative algebras}, recently discovered by Miquel, are combinatorial structures unifying classical and intuitionistic realizability as well as forcing.  In this paper we introduce \emph{implicative assemblies} as sets valued in the separator of an underlying implicative algebra. Given a fixed implicative algebra $\A$, implicative assemblies over $\A$ organise themselves in a category $\AsmA$ with \emph{tracked} set-theoretical functions as morphisms. We show that $\AsmA$ is a quasitopos with NNO.
\end{abstract}

\nocite{Kle45,Coh63,Coh64,HJP80,Hyl82,Str13}
\section{Introduction}\label{s:Intro}

As for other parts of the logical apparatus, it is quite practical to handle structural aspects of realizalibility and those of forcing in terms of toposes or categories with similar features. It is fair to say that sheaf models of intuitionistic logic \cite{fourman2006sheaves} (used as environment for constructive analysis \cite{fourman2006sheaf}) were a precursor of sorts here, as further study of the phenomenon lead to the discovery of the \emph{effective topos} \cite{Hyl82}. It was the first of an array of toposes encompassing diverse notions of intuitionistic realizability \cite{Oos08}.
In a different direction, sheaf models were also the starting point for topos-theoretical versions of forcing. While topos-theoretical independence proofs for the continuum hypothesis \cite{lawvere1970quantifiers,tiernay1972sheaf} or for Suslin's hypothesis \cite{bunge1974topos} were post factum, other independence proofs were  carried out directly by means of topos-theoretic techniques \cite{fourman1982world,freyd1980axiom}. On the other hand, classical realizability's origin is quite distinct. The observation that extending the $\lambda$-calculus with control operators yields proof terms for classical logic \cite{griffin1989formulae} lead to the concept of Krivine's machine, the computational device underpinning classical realizability \cite{Kri09} and it's topos-theoretical counterpart \cite{Str13}.

It has been observed every now and again that realizability and forcing exhibit
common traits, yet a unified foundation for these anything but similar topics has
remained elusive for decades. Still, recent years have witnessed a significant change
in this state of things with the discovery of implicative algebras  \cite{Miq20},
which turned out to be the basic combinatorial structures providing the
missing link. An implicative algebra is a complete lattice $(\A,\ile)$ equipped with a
(non-necessarily Heyting) implication operator $\to$ and a subset $\S$ called
separator.  The meaning of its elements is twofold: they can be
thought of as (generalised) realizers on one hand and as truth values
on the other. The fragment $(\A, \ile, \to)$ of an implicative algebra, called
implicative structure, is a model and a Curry-style type system for a
$\lambda$-calculus with constants in $\A$, possibly extended by the control
operator \texttt{call/cc}.

The order $\ile$ can be
thought of as a (sub)typing relation, so a deduction establishing a type is
tantamount to the construction of a realizer.
However,
with this infrastructure alone {\tmem{any}} element of $\A$ can be
thought of as a realizer, which is not really an option given the Curry-Howard
correspondence. It is here where the separator $\S$ comes into play:
it is the device which carves out the {\tmem{correct}} realizers.
Implicative algebras are in their essence abstract versions of
classical realizability algebras \cite{Kri11}, yet their additional layers of abstraction also allow to accommodate intuitionistic
realizability. What is more, forcing fits into the framework as well.

Since implicative algebras unify the notions of (intuitionistic and classical) realizability and forcing, to ask how they translate into categorical logic is a rather natural question. Similarly to the intuitionistic case \cite{Hyl82}, an implicative algebra
$(\A, \ile, \to, \S)$ gives rise to
{\tmem{non-standard predicates}}, that is functions $\phi : I \to
\A$. The set $\A^I$ of non-standard predicates on $I$
equipped with the pointwise implication gives rise to a preorder
$\ent_{\S[I]}$ given by
\begin{eqnarray*}
  \phi \ent_{\S[I]} \psi & \liff & \left( \bigimeet_{i \in
  I} \phi (i) \to \psi (i) \right) \in \S
\end{eqnarray*}
This preorder is Heyting. The non-standard predicates neatly organise
themselves in a tripos \cite{pitts2002tripos} when varying the index set $I$. It is a truly
outstanding fact that {\tmem{any}} (set-based) tripos is equivalent to one
built from an implicative algebra \cite{Miq20}. Applying the functor $\Set
[-]$ (colloquially known as {\tmem{tripos-to-topos construction}} \cite{Pit81})
to such an {\tmem{implicative tripos}} gives rise to an {\tmem{implicative
topos}}. Analysing the construction unveils the objects of an implicative
topos as sets
equipped with an $\A$-valued
{\tmem{non-standard equality}},
symmetric and transitive but not necessarily reflexive in the internal language
of the underlying implicative tripos (the lack of reflexivity is perhaps the main point of the construction). We call such sets \emph{implicative}. Morphisms among implicative sets are functional relations in the
internal language of the underlying tripos.

The
non-standard equality
$| - \!\! \approx \!\! - | : X \times X \to
\A$ on an implicative set $X$ is reflexive when
$| x \!\approx \! x | \in \S$ for
all $x \in X$. As the latter condition generally fails, $\exi{X} (x) \assign | x \! \approx \! x |$ acts as an
{\tmem{existence predicate}}. We call {\tmem{ghosts}} elements $x \in X$ such
that $\exi{X}(x) \nin \S$. Within an implicative topos sits the
subcategory of {\tmem{implicative assemblies}}, which are implicative sets
with no ghost elements. For this reason, morphisms of implicative assemblies
amount to plain set-theoretical functions verifying a {\tmem{tracking condition}}. It is
possible to construct the category of implicative assemblies from scratch
independently of the implicative topos. The construction itself is all but
obvious: objects are pairs $(X, \exi{X})$ with $X$ a set and $\exi{X} : X \to
\S$ an existence predicate while morphisms $f : (X, \exi{X}) \to
(Y, \exi{Y})$ are set-theoretical functions $f : X \to Y$ verifying the tracking
condition
\[ \bigimeet_{x \in X} (\exi{X} (x) \to \exi{Y} (f (x))) \in \S
\]
All the same, it is  not immediate from the onset that the category of implicative assemblies is a quasitopos. In this paper we address the question and show that it is a quasitopos with NNO, leaving the intriguing questions
related to the nature of associated completions \cite{robinson1990colimit,menni2000exact} for the follow-up one.

Section 2 contains a thorough introduction to implicative algebras, section 3
defines the category of implicative assemblies, section 4 details its
structure of quasitopos with NNO while section 5 addresses the particular case of
implicative quasitoposes built from an implicative algebra of forcing type.

\section{Implicative algebras: a primer}
\label{s:ImpAlg}

\subsection{Implicative structures}
The notion of implicative algebra is based on the notion of
implicative structure, that it is now time to present.
Formally, an \emph{implicative structure} is a triple
$(\A,{\ile},{\to})$ where $(\A,{\ile})$ is a complete lattice (whose
ordering $\ile$ represents \emph{subtyping}), and where ${\to}$ is a
binary operation on~$\A$ (the \emph{implication} of~$\A$) such that
\begin{enumerate}
\item[(1)] If $a'\ile a$ and $b\ile b'$,
  then $(a\to b)\ile(a'\to b')$\hfill($a,a',b,b'\in\A$)
\item[(2)] $\ds a\to\bigimeet_{\mcent{b\in B}}b
  ~=~\bigimeet_{\mcent{b\in B}}(a\to b)$\hfill($a\in A$, $B\subseteq\A$)
\end{enumerate}

\begin{examples}\label{ex:ImpStruct}
  There are many examples of implicative structures, mostly coming
  from the theory of forcing and from intuitionistic and classical 
  realizability.
  For instance:
  \begin{enumerate}
  \item[(1)] Each complete Heyting (or Boolean) algebra $(\A,{\ile})$
    is an implicative structure, whose implication $\to$ is derived
    from the ordering $\ile$ using Heyting's adjunction:
    $$c\imeet a\ile b\qquad\text{iff}\qquad c\ile a\to b
    \eqno(a,b,c\in H)$$
  \item[(2)] If $(P,{\cdot})$ is a (total) applicative structure, then
    the triple $(\A,{\ile},{\to})$ defined by
    $$\begin{array}{r@{~{}~}c@{~{}~}l}
      \A&:=&\Pow(P) \\
      a\ile b&:\equiv&a\subseteq b \\
      a\to b&:=&\{x\in P~:~\forall y\in q,~x\cdot y\in b\}\\
    \end{array}\eqno\begin{array}{r@{}}
    \\(a,b\in\A)\\(\text{Kleene's implication})\\
    \end{array}$$
    is an implicative structure.
  \item[(3)] Each \emph{classical realizability algebra}~\cite{Kri11}
    (also called \emph{abstract Krivine structure} in~\cite{Str13})
    induces an implicative structure as well, and even a classical
    implicative algebra~\cite{Miq20}.
  \end{enumerate}
\end{examples}

\subsubsection*{Elements of~$\A$ as truth values}
The elements of an implicative structure~$\A$ are primarily intended
to represent ``truth values'' according to some semantic of
(intuitionistic or classical) logic.
Although implicative structures focus on implication (modeled
by the operator $\to$) and universal quantification (modeled by
arbitrary meets), they also provide all the standard connectives and
quantifiers of logic, using the standard second-order encodings:
\begin{center}\medbreak
  $\begin{array}{@{}l@{\qquad}rcl}
    (\text{Absurdity}) & \bot&:=&\min\A \\[6pt]
    (\text{Triviality}) & \top&:=&\max\A \\[6pt]
    (\text{Negation}) & \lnot a&:=&a\to\bot \\[6pt]
    (\text{Conjunction}) & a\iand b &:=& \ds\bigimeet_{\mcent{c\in\A}}
    \bigl((a\to b\to c)\to c\bigr) \\
    (\text{Disjunction}) & a\ior b&:=&\ds\bigimeet_{\mcent{c\in\A}}
    \bigl((a\to c)\to(b\to c)\to c\bigr) \\
    (\text{Universal quantification}) &
    \ds\bigforall_{i\in I}a_i&:=&\ds\bigimeet_{i\in I}a_i \\
    (\text{Existential quantification}) &
    \ds\bigexists_{i\in I}a_i&:=&\ds
    \bigimeet_{\mcent{c\in\A}}\biggl(
    \Bigl(\bigimeet_{\mcent{i\in I}}(a_i\to c)\Bigr)
    ~\to~c\biggr) \\
  \end{array}$\hfill~\medbreak
\end{center}
(for all $a,b\in\A$, $(a_i)_{i\in I}\in\A^I$).

\subsubsection*{Elements of $\A$ as generalized realizers}
However, the salient feature of implicative structures is that their
elements can also be used to models proofs---or better:
realizers---using an interpretation of the (pure) $\lambda$-calculus
that is defined as follows:
\begin{itemize}
\item[$\bullet$] Given elements $a,b\in\A$, we define their
  \emph{application} $ab\in\A$ by:
  $$ab~:=~\bigimeet\bigl\{c\in\A:a\ile(b\to c)\bigr\}\,.$$
\item[$\bullet$] Given an arbitrary function $f:\A\to\A$, we define
  its \emph{abstraction} $\ilambda{f}\in\A$ by:
  $$\ilambda{f}~:=~\bigimeet_{\mcent{a\in\A}}(a\to f(a))\,.$$
\end{itemize}

\begin{fact}[Properties of application and abstraction]
  For all elements $a,a',$ $b,b',c\in\A$ and for all functions
  $f,f':\A\to\A$: 
  \begin{enumerate}
  \item If $a\ile a'$ and $b\ile b'$, then $ab\ile a'b'$\hfill
    (Application is monotonic)
  \item If $f\ile f'$ (pointwise), then
    $\ilambda{f}\ile\ilambda{f'}$\hfill 
    (Abstraction is monotonic)
  \item$(\ilambda{f})a~\ile~f(a)$\hfill($\beta$-reduction)
  \item$a~\ile~\ilambda{f}(x\mapsto ax)$\hfill($\eta$-expansion)
  \item $ab\ile c$\quad iff\quad $a\ile(b\to c)$\hfill(Adjunction)
  \end{enumerate}
\end{fact}

Thanks to the above two operations, we can interpret each closed
$\lambda$-term~$t$ with parameters%
\footnote{Recall that a \emph{$\lambda$-term with parameters in~$\A$}
  is a term of the pure $\lambda$-calculus possibly enriched with
  (computationally inert) constants taken in the set~$\A$.}
in~$\A$ as the element $(t)^{\A}\in\A$ defined by:
$$\begin{array}{rcl}
  (a)^{\A}&:=&a\\
  (tu)^{\A}&:=&t^{\A}u^{\A} \\
  (\Lam{x}{t})^{\A}&:=&
  \ilambda\Bigl(a\mapsto(t[x:=a])^{\A}\Bigr) \\
\end{array}$$

\begin{fact}
  For all $\lambda$-terms $t,t'$ with parameters in~$\A$:
  \begin{enumerate}
  \item If~~$t\redd_{\beta}t'$,~~then~~$(t)^{\A}\ile(t')^{\A}$\hfill
    ($\beta$-reduction)
  \item If~~$t\redd_{\eta}t'$,~~then~~$(t)^{\A}\ige(t')^{\A}$\hfill
    ($\eta$-reduction)
  \end{enumerate}
\end{fact}

Note that that the interpretation function $t\mapsto(t)^{\A}$ is not
necessarily injective.
For example, in the particular case where $\A$ is a complete Heyting
algebra (Example~\ref{ex:ImpStruct}~(1)), it is easy to check that
application collapses to the meet (namely: $ab=a\imeet b$ for all
$a,b\in\A$), and that all pure $\lambda$-terms~$t$ (without
parameters) are interpreted as the top truth value: $(t)^{\A}=\top$.
On the other end of the spectrum, when~$\A$ is induced by a
non-degenerated (total) applicative structure
(Example~\ref{ex:ImpStruct}~(2)) that is also a combinatory algebra,
the interpretation of pure $\lambda$-terms (without parameters) is
necessarily injective on $\beta\eta$-normal forms.

\subsubsection*{The ambivalence of the ordering $a\ile b$}
From what precedes, it is clear that both the terms and the types of a
polymorphic $\lambda$-calculus such as Curry-style system~F (possibly
enriched with finite intersection types) can be interpreted in any
implicative structure~$\A$.
In this setting, it is easy to check that if some pure 
$\lambda$-term~$t$ has some type~$U$ in the aforementioned type
system, then we have $(t)^{\A}\ile(U)^{\A}$, independently from the
implicative structure~$\A$.
(Technically, this result is proved by introducing a suitable notion
of \emph{semantic subtyping}, see~\cite{Miq20}, \Sect~2.5.)

The above observation shows that the partial ordering $\ile$ at the
core of implicative structures (and implicative algebras) can be given
different meanings depending on whether we consider its arguments as
types or realizers.
So that $a\ile b$ may actually read:
\begin{itemize}
\item \emph{$a$ is a subtype of~$b$},
  if we view both~$a$ and~$b$ as types;
\item\emph{$a$ has type~$b$},
  if we view~$a$ as a realizer and~$b$ as a type;
\item \emph{$a$ is more defined than~$b$},
  if we view both~$a$ and~$b$ as realizers.
\end{itemize}

In what follows, we shall frequently mention the (semantic)
combinators $\mathbf{K}^{\A}=(\Lam{xy}{x})^{\A}$ and
$\mathbf{S}^{\A}=(\Lam{xyz}{xz(yz)})^{\A}$, observing that both
combinators actually coincide with their respective principal types
in~$\A$, in the sense that:
\begin{fact}[Combinators $\mathbf{K}$ and $\mathbf{S}$]
  In any implicative structure~$\A$, we have
  $$\begin{array}[b]{rcl}
    \mathbf{K}^{\A}&=&
    \ds\bigimeet_{\mcent{a,b\in\A}}(a\to b\to a)\\
    \noalign{\smallskip}
    \mathbf{S}^{\A}&=&\ds\bigimeet_{\mcent{a,b,c\in\A}}
    \bigl((a\to b\to c)\to(a\to b)\to a\to c\bigr)\\
  \end{array}\leqno\text{and}$$
\end{fact}

\subsubsection*{Interpreting the control operator `call$/$cc'}
The above interpretation of $\lambda$-terms in the implicative
structure~$\A$ extends to the operator~$\cc$ (`call$/$cc', for `call
with current continuation'), provided we identify the latter with
Peirce's law:
$$(\cc)^{\A}~:=~\bigimeet_{\mcent{a,b\in\A}}(((a\to b)\to a)\to a)
~=~\bigimeet_{a\in\A}((\lnot a\to a)\to a)\,.$$
Note however that in most implicative structures (typically those
coming from intuitionistic realizability), the above definition
collapses to~$\bot$.

\subsection{Separators and implicative algebras}
Implicative structures do not come with a fixed criterion of truth,
telling us which elements of~$\A$ are considered to be true.
Technically, such a criterion is provided by a \emph{separator}, that
is a subset $\S\subseteq\A$ such that: 
\begin{enumerate}
\item If $a\in\S$ and $a\ile b$, then $b\in\S$\hfill
  (upwards closed)
\item $\mathbf{K}^{\A}\in\S$ and $\mathbf{S}^{\A}\in\S$\hfill
  (Hilbert axioms)
\item If $a\to b\in\S$ and $a\in \S$, then $b\in\S$\hfill
  (modus ponens)
\end{enumerate}
writing $\mathbf{K}^{\A}=(\Lam{xy}{x})^{\A}$ and
$\mathbf{S}^{\A}=(\Lam{xyz}{xz(yz)})^{\A}$.
A separator $\S$ is \emph{consistent} (resp.\ \emph{classical}) when
$\bot\notin\S$ (resp.\ when $\cc^{\A}\in\S$).
An \emph{implicative algebra} is just an implicative structure
$(\A,{\ile},{\to})$ equipped with a separator $\S\subseteq\A$.
An implicative algebra $(\A,{\ile},{\to},\S)$ is \emph{consistent}
(resp.\ \emph{classical}) when the underlying separator
$\S\subseteq\A$ is.

Separators enjoy the following closure properties:
\begin{fact}
  For all elements $a,b\in\A$ and for all separators
  $\S\subseteq\A$:
  \begin{enumerate}
  \item If $a\in \S$ and $b\in \S$, then $ab\in \S$
  \item If $a\in \S$ and $b\in \S$, then $a\times b\in \S$
  \item If $a\in \S$ or $b\in \S$, then $a+b\in \S$
  \item If~$t$ is a closed $\lambda$-term with parameters in~$\S$,
    then $(t)^{\A}\in\S$.
  \end{enumerate}
\end{fact}\noindent
(Last item intuitively expresses that all `proofs'---possibly with
`axioms' in~$\S$---are `true'.)

In the particular case where the implicative structure~$\A$ is a
complete Heyting algebra (Example~\ref{ex:ImpStruct}~(1)), it can be
shown that a separator is the same as a filter.
However, in the general case, separators are \emph{not} filters, for
that they are not closed under binary meets (in general).

\subsection{The implicative tripos}
In any implicative algebra $\A=(\A,{\ile},{\to},\S)$, the separator
$\S$ induces a preorder of \emph{entailment} on the set~$\A$, written
$a\ent_{\S}b$ and defined by
$$a\ent_{\S}b~:\equiv~(a\to b)\in\S
\eqno(a,b\in\A)$$
It follows from the definition of the notion of a separator that the
preordered set $(\A,{\ent_{\S}})$ is a pre-Heyting algebra (and even
a pre-Boolean algebra when the separator~$\S$ is classical).

More generally, for each set~$I$ the separator~$\S$ induces a preorder
of entailment on the set $\A^I$ of $I$-indexed families of truth
values, written $\ent_{\S[I]}$ and defined by
$$a\ent_{\S[I]}b~:\equiv~\bigimeet_{i\in I}(a_i\to b_i)\in\S
\eqno(a,b\in\A^I)$$
Again, the preordered set $(\A^I,{\ent_{\S[I]}})$ is a pre-Heyting (or
pre-Boolean) algebra.
Moreover if can be shown~\cite{Miq20} that:
\begin{theorem}[Implicative tripos]
  For each implicative algebra $\A=(\A,{\ile},{\to},\S)$,
  the correspondence mapping each set~$I$ to the pre-Heyting algebra
  $(\A^I,\S[I])$ is functorial and constitutes a $\Set$-based tripos,
  in the sense of Hyland, Johnstone and Pitts~\cite{HJP80}.
\end{theorem}

As shown in~\cite{Miq20}, the above construction encompasses all known
triposes known so far: forcing triposes (induced by complete Heyting
algebras), intuitionistic realizability triposes (induced by total and
even partial combinatory algebras, using some completion trick) as
well as classical realizability triposes (induced by abstract Krivine
structures).

Actually, we can even prove~\cite{Miq20b} that:
\begin{theorem}[Completeness]\label{th:Completeness}
  Each $\Set$-based tripos (in the sense of Hyland, Johnstone and
  Pitts~\cite{HJP80}) is isomorphic to an implicative tripos.
\end{theorem}

\section{Construction of the category $\AsmA$}
\label{s:AsmA}

From now on, $\A=(\A,{\ile},{\to},\S)$ denotes a given implicative
algebra.

\subsection{Definition of the category~$\AsmA$}
From the implicative algebra $\A$, we define the category $\AsmA$ of
\emph{assemblies} on~$\A$ as follows:
\begin{itemize}
\item The objects of $\AsmA$ are the \emph{assemblies} on~$\A$,
  namely:
  the pairs $X=(\Car{X},\Ex{X})$ formed by
  an arbitrary set $\Car{X}$, called the \emph{carrier} of~$X$,
  and a function $\Ex{X}:\Car{X}\to\S$,
  called the \emph{existence predicate} of~$X$,
  that associates to each element $x\in\Car{X}$ a truth value
  $\Ex{X}(x)$ (in the separator $\S\subseteq\A$) that intuitively
  `certifies' the existence of~$x$ in the assembly~$X$.
\item Given two assemblies $X=(\Car{X},\Ex{X})$ and
  $Y=(\Car{Y},\Ex{Y})$, the morphisms from~$X$ to~$Y$ are the
  set-theoretic functions $f:\Car{X}\to\Car{Y}$ satisfying the
  \emph{tracking condition} for~$f$:
  $$\bigimeet_{x\in X}\bigl(\Ex{X}(x)\to\Ex{Y}(f(x))\bigr)~\in~\S\,.$$
  Note that in order to prove the above condition, it is sufficient
  (and obviously necessary) to exhibit an element $s_f\in\S$ ---which we
  shall call a \emph{tracker} of~$f$--- such that
  $$s_f~\ile~\bigimeet_{x\in X}\bigl(\Ex{X}(x)\to\Ex{Y}(f(x))\bigr)\,.$$
  (In practice, we shall provide most trackers as `proofs' of the
  form $s_f:=(t)^{\A}$, where~$t$ is a pure $\lambda$-term constructed
  following the Curry-Howard correspondence.)
\end{itemize}

Given assemblies $X$, $Y$, $Z$ and morphisms $f:X\to Y$ and $g:Y\to Z$
respectively tracked by elements $s_f,s_g\in\S$, we easily check that
$$(\Lam{x}{s_g\,(s_f\,x)})^{\A}~\ile~
\bigimeet_{x\in X}\bigl(\Ex{X}(x)\to\Ex{Z}((g\circ f)(x))\bigr)\,,$$
so that $g\circ f:X\to Z$ is also a morphism.
Finally, we have
$$(\Lam{x}{x})^{\A}~\ile~
\bigimeet_{x\in X}\bigl(\Ex{X}(x)\to\Ex{X}(\id_X(x))\bigr)\,,$$
hence $\id_X:X\to X$ is a morphism as well.
This makes clear that $\AsmA$ is a category.

Recall that in the category of assemblies, morphisms are just
set-theoretic maps satisfying the corresponding tracking condition.
Therefore, two morphisms are equal \emph{as morphisms} in the
category~$\AsmA$ if and only if they are equal as set-theoretic
functions.

\subsection{The adjoint functors $\Gamma:\AsmA\to\Set$
  and $\Delta:\Set\to\AsmA$}
The category of assemblies on~$\A$ naturally comes with a forgetful
functor $\Gamma:\AsmA\to\Set$ mapping each assembly~$X$ to its carrier
$\Gamma{X}:=\Car{X}$ and each morphism $f:X\to Y$ to itself
as a set-theoretic function, that is:
$\Gamma{f}:=f:\Car{X}\to\Car{Y}$.
(So that by construction, the functor $\Gamma$ is faithful.)

The functor $\Gamma:\AsmA\to\Set$ is actually the left adjoint of a
functor $\Delta:\Set\to\AsmA$ that can be defined as follows:
\begin{itemize}
\item For each set $X$, we define $\Delta{X}$ as the assembly whose
  carrier is $\Car{\Delta{X}}:=X$ and whose existence predicate is the
  constant function $\Ex{\Delta{X}}:=(\_\mapsto\top)$ mapping each
  element of~$X$ to the truth value $\top\in\S$.
\item For each set-theoretic map $f:X\to Y$, we define the morphism
  $\Delta{f}:\Delta{X}\to\Delta{Y}$ as the map~$f$ itself, observing
  that such a map is tracked by the truth value $\top\in\S$.
\end{itemize}
By construction, the functor $\Delta:\Set\to\AsmA$ is full and
faithful.

\begin{proposition}\label{p:AdjGammaDelta}
  We have the adjunction:
  $\mskip35mu\xymatrix@C=12pt{
    \mskip-35mu\AsmA\ar@/^1pc/[rr]^{\ds\Gamma}&\bot&
    \Set\ar@/^1pc/[ll]^{\ds\Delta}\\
  }$.
\end{proposition}

\begin{proof}
  Given objects $X\in\AsmA$ and $Y\in\Set$, it suffices to observe
  that each set-theoretic map $f:\Car{X}\to\Car{\Delta{Y}}$ is tracked
  by the truth value $\top\in\S$, hence we have
  $$\Hom_{\AsmA}(X,\Delta{Y})~=~\Hom_{\Set}(\Gamma{X},Y)\,.
  \eqno\mbox{\qedhere}$$
\end{proof}

From the above adjunction, it is clear that $\Gamma:\AsmA\to\Set$
preserves all colimits that turn out to exist in~$\AsmA$, whereas
$\Delta:\Set\to\AsmA$ preserves all limits in $\Set$.

\subsection{Characterizing monomorphisms and epimorphisms}

\begin{proposition}[Monos and epis in $\AsmA$]\label{p:MonoEpi}
  In the category $\AsmA$, an arrow $f:X\to Y$
  is mono (resp.\ epi) if and only if it is injective
  (resp.\ surjective) as a set-theoretic map.
\end{proposition}

\begin{proof}
  Let $X$ and~$Y$ be assemblies.
  Since each morphism $f:X\to Y$ is also a set-theoretic map
  $f:\Car{X}\to\Car{Y}$, it is clear that if~$f$ is injective
  (resp.\ surjective), then it is mono (resp.\ epi) as a morphism
  in~$\AsmA$.
  Let us now prove the converse implications:
  \begin{itemize}
  \item Suppose that $f:X\to Y$ is mono in $\AsmA$.
    Given $x_1,x_2\in\Car{X}$ such that $f(x_1)=f(x_2)$
    ($\in\Car{Y}$), we want to prove that $x_1=x_2$.
    For that, we write $1=\{*\}$ the terminal object of~$\Set$,
    and consider the morphisms $h_1,h_2:\Delta{1}\to X$ defined
    by $h_1(*)=x_1$ and $h_2(*)=x_2$.
    (The maps $h_1,h_2:1\to\Car{X}$ are tracked by the truth values
    $(\top\to\Ex{X}(x_1))\in\S$ and $(\top\to\Ex{X}(x_2))\in\S$,
    respectively.)
    Since $f(x_1)=f(x_2)$, we get
    $f\circ h_1=f\circ h_2$ hence $h_1=h_2$ in $\AsmA$ (since $f$ is
    mono) and therefore $x_1=x_2$.
  \item Suppose that $f:X\to Y$ is epi in $\AsmA$.
    We want to prove that $f:\Car{X}\to\Car{Y}$ is surjective, which
    amounts to prove that~$f$ is epi in $\Set$.
    For that, consider a set~$Z$ with set-theoretic maps
    $h_1,h_2:\Car{Y}\to Z$ such that $h_1\circ f=h_2\circ f$.
    Using the equality
    $\Hom_{\Set}(\Car{Y},Z)=\Hom_{\AsmA}(Y,\Delta{Z})$
    (cf\ proof of Prop.~\ref{p:AdjGammaDelta}), we observe
    that these set-theoretic maps are also morphisms
    $h_1,h_2:Y\to\Delta{Z}$ in $\AsmA$, hence $h_1=h_2$ in $\AsmA$
    (since~$f$ is epi), and therefore $h_1=h_2$ in~$\Set$.\qedhere
  \end{itemize}
\end{proof}

However, the category $\AsmA$ is not balanced in general, since there
are morphisms that are both mono and epi without being isomorphisms.
For a counter-example, consider the two-point
assembly~$\mathbf{2}$ defined by%
\footnote{The reader is invited to check that in the category $\AsmA$,
  the object $\mathbf{2}$ is the coproduct $\mathbf{1}+\mathbf{1}$
  (i.e.\ the type of Booleans), writing $\mathbf{1}:=\Delta{1}$ the
  terminal object of the category.}
$$\begin{array}{r@{~{}~}c@{~{}~}l@{~{}~}c@{~{}~}l@{~{}~}c@{~{}~}l}
  \Car{\mathbf{2}}&:=&\multicolumn{3}{l}{2~=~\{0,1\}} \\[6pt]
  \Ex{\mathbf{2}}(0)&:=&(\Lam{xy}{x})^{\A}&=&
  \ds\bigimeet_{\mcent{a,b\in\A}}(a\to b\to a)&\in&\S \\[3pt]
  \Ex{\mathbf{2}}(1)&:=&(\Lam{xy}{y})^{\A}&=&
  \ds\bigimeet_{\mcent{a,b\in\A}}(a\to b\to b)&\in&\S \\
\end{array}$$
as well as the morphism $i:\mathbf{2}\to\Delta{2}$ whose underlying
set-theoretic function $i:2\to 2$ is the identity function (tracked
by $\top\in\S$).

\begin{proposition}
  The morphism $i:\mathbf{2}\to\Delta{2}$ is always mono and epi
  (independently from the implicative algebra~$\A$), but it is an
  isomorphism if and only if the separator $\S\subseteq\A$ is a
  filter.
\end{proposition}

\begin{proof}
  It is clear that the morphism $i:\mathbf{2}\to\Delta{2}$, that is
  both mono and epi in~$\AsmA$ (since bijective in~$\Set$), is an
  isomorphism if and only if the inverse set-theoretic function
  $\id_{2}:2\to 2$ is tracked as a morphism of type
  $\Delta{2}\to\mathbf{2}$ in $\AsmA$.
  For that, we observe that
  $$\bigimeet_{x\in 2}\bigl(\Ex{\Delta{2}}(x)\to\Ex{\mathbf{2}}(x)\bigr)
  ~=~\top\to\bigimeet_{x\in 2}\Ex{\mathbf{2}}(x)
  ~=~\top\to\Fork^{\A}\,,$$
  writing $\Fork^{\A}:=\bigimeet_{a,b\in\A}(a\to b\to a\imeet b)
  ~=~(\Lam{xy}{x})^{\A}\imeet(\Lam{xy}{y})^{\A}$
  the non-deterministic choice operator in~$\A$~\cite[\Sect\,3.7]{Miq20}.
  We conclude observing that
  $$\begin{array}[b]{r@{\qquad}c@{\qquad}l}
    \ds\bigimeet_{x\in 2}
    \bigl(\Ex{\Delta{2}}(x)\to\Ex{\mathbf{2}}(x)\bigr)
    ~\in~\S
    &\text{iff}& \top\to\Fork^{\A}~\in~\S \\[-8pt]
    &\text{iff}& \Fork^{\A}~\in~\S \\[3pt]
    &\text{iff}& \S~\text{is a filter
      (from \cite[Prop. 3.27]{Miq20}).} \\
  \end{array}\eqno\mbox{\qedhere}$$
\end{proof}

\begin{remark}
  Recall that in implicative algebras, separators are in general not
  filters.
  A typical counter-example is the (intuitionistically consistent)
  implicative algebra defined by:
  \begin{itemize}
  \item $\A:=\Pow(\Lambda)$\quad(writing~$\Lambda$ the set
    of closed $\lambda$-terms up to $\beta$-conversion)
  \item $a\ile b~:\equiv~a\subseteq b$\quad
    (inclusion of sets of $\lambda$-terms)
  \item $a\to b~:=~\{t\in\Lambda\mid
    \forall u\in a,~tu\in b\}$\quad (Kleene arrow)
  \item $\S~:=~\{a\in\A\mid a~\text{inhabited}\}$
  \end{itemize}
  To show that the separator~$\S$ is not a filter, take
  $a:=\{\Lam{xy}{x}\}\in\S$ and $b:=\{\Lam{xy}{y}\}\in\S$,
  and observe that $a\cap b=\varnothing$
  (from the Church-Rosser property), hence
  $a\cap b\notin\S$.
\end{remark}

\section{Structure of quasi-topos}
\label{s:QTopos}

The aim of this section is to prove that the category $\AsmA$ of
assemblies is a quasi-topos, namely: that the category $\AsmA$ has all
finite limits (\Sect~\ref{ss:FinLim}) and colimits
(\Sect~\ref{ss:FinColim}), that it is locally Cartesian closed
(\Sect~\ref{ss:LCCC}) and has a strong subobject classifier
(\Sect~\ref{ss:SSOC}).
Moreover, we shall prove that this quasi-topos has a natural numbers
object (\Sect~\ref{ss:NNO}).

\subsection{Finite limits}\label{ss:FinLim}
Let us first prove that $\AsmA$ has all finite limits.
\begin{proposition}\label{p:Terminal}
  The assembly $\mathbf{1}:=\Delta{1}$ (where $1=\{*\}$)
  is the terminal object of~$\AsmA$.
\end{proposition}

\begin{proof}
  Indeed, for each assembly~$X$, we observe that
  $$\Hom_{\AsmA}(X,\Delta{1})~=~
  \Hom_{\Set}(\Gamma{X},1)~=~\{(\_\in\Car{X}\mapsto{*})\}\,.
  \eqno\mbox{\qedhere}$$
\end{proof}

Now, given assemblies~$A$ and~$B$, we define the \emph{product assembly}
$A\times B$, letting:
\begin{itemize}
\item$\Car{A\times B}~:=~\Car{A}\times\Car{B}$;
\item$\Ex{A\times B}(a,b)~:=~\Ex{A}(a)\iand\Ex{B}(b)$\ \
  ($\in\S$)\quad for all $(a,b)\in\Car{A\times B}$.
\end{itemize}
We easily check that the set-theoretic projections
$\pi_1^{A,B}:\Car{A\times B}\to\Car{A}$ and
$\pi_2^{A,B}:\Car{A\times B}\to\Car{B}$
are actually morphisms of types $A\times B\to A$ and $A\times B\to B$
(respectively) in~$\AsmA$, since they are tracked by the
truth values $(\Lam{z}{z\,(\Lam{xy}{x})})^{\A},
(\Lam{z}{z\,(\Lam{xy}{y})})^{\A}\in\S$
(respectively).
\begin{proposition}[Binary product]\label{p:BinProd}
  For all assemblies $A$ and~$B$, the product assembly~$A\times B$
  equipped with the two morphisms $\pi_1^{A,B}:A\times B\to A$ and
  $\pi_2^{A,B}:A\times B\to B$ is the binary product of~$A$ and~$B$
  in the category~$\AsmA$.
\end{proposition}

\begin{proof}
  Given an assembly~$X$ and morphisms $f:X\to A$ and $g:X\to B$,
  we want to show that there is a unique morphism $h:X\to A\times B$
  such that $\pi_1^{A,B}\circ h=f$ and $\pi_2^{A,B}\circ h=g$:
  $$\xymatrix@C=48pt@R=48pt{
    &X\ar[dl]_{\ds f}\ar[dr]^{\ds g}
    \ar@{-->}[d]^{\ds h}&\\
    A&A\times B\ar[l]^{\ds\pi_1^{A,B}}\ar[r]_{\ds\pi_2^{A,B}}&B \\
  }$$
  \smallbreak\noindent
  \emph{Existence of~$h$.}\quad
  Consider the map
  $h:\Car{X}\to\Car{A}\times\Car{B}$ defined by
  $h(x):=(f(x),g(x))$ for all $x\in\Car{X}$.
  Given trackers $s_f,s_g\in\S$ of the morphisms $f:X\to A$ and
  $g:X\to B$ (resp.), we observe that
  $$(\Lam{xz}{z\,(s_f\,x)\,(s_g\,x)})^{\A}~\ile~
  \bigimeet_{x\in X}\bigl(\Ex{X}(x)\to\Ex{A\times B}(h(x))\bigr)\,,$$
  which proves that~$h$ is a morphism of type $X\to A\times B$.
  And from the definition of~$h$, it is clear that
  $\pi_1^{A,B}\circ h=f$ and $\pi_2^{A,B}\circ h=g$.
  \smallbreak\noindent
  \emph{Uniqueness of~$h$.}\quad
  Obvious from the equalities $\pi_1^{A,B}\circ h=f$ and
  $\pi_2^{A,B}\circ h=g$.
\end{proof}

Finally, given parallel arrows $f,g:A\to B$ (in the category~$\AsmA$),
we define a new assembly $\Ker(f,g)$, letting:
\begin{itemize}
\item$\Car{\Ker(f,g)}~:=~\{a\in\Car{A}\mid f(a)=g(a)\}$;
\item$\Ex{\Ker(f,g)}(a)~:=~\Ex{A}(a)$\ \ ($\in\S$)\quad
  for all $a\in\Car{\Ker(f,g)}$.
\end{itemize}
We also consider the monomorphism $k:\Ker(f,g)\to A$ induced by the
inclusion $\Car{\Ker(f,g)}\subseteq\Car{A}$ and tracked by the
truth value $(\Lam{x}{x})^{\A}\in\S$.

\begin{proposition}[Equalizer]\label{p:Equalizer}
  For all parallel arrows $f,g:A\to B$ in~$\AsmA$, the assembly
  $\Ker(f,g)$ equipped with the arrow $k:\Ker(f,g)\to A$
  is the equalizer of~$f$ and~$g$ in~$\AsmA$.
\end{proposition}

\begin{proof}
  Given an assembly~$X$ and a morphism $h:X\to A$ such that
  $f\circ h=g\circ h$, we want to show that there is a unique
  morphism $h':X\to\Ker(f,g)$ such that $k\circ h'=h$:
  $$\xymatrix@C=40pt@R=40pt{
    X\ar@{-->}[d]_{\ds h'}\ar[dr]^{\ds h} \\
    \Ker(f,g)\ar[r]_{\ds k} &
    A\ar@<2pt>[r]^{\ds f}\ar@<-2pt>[r]_{\ds g} & B \\
  }$$
  \smallbreak\noindent
  \emph{Existence of~$h'$.}\quad
  Since $f\circ h=g\circ h$ (set-theoretically), we have
  $\Im(h)\subseteq\Car{\Ker(f,g)}$.
  Hence we can define $h':X\to\Ker(f,g)$ letting
  $h'(x):=h(x)$ for all $x\in\Car{X}$, and taking the same tracker
  as for $h:X\to A$.
  We obviously have $k\circ h'=h$ by construction.
  \smallbreak\noindent
  \emph{Uniqueness of~$h'$.}\quad
  Obvious from the equality $k\circ h'=h$.
\end{proof}

From Prop.~\ref{p:Terminal}, \ref{p:BinProd} and~\ref{p:Equalizer},
it is clear that:
\begin{proposition}[Finite limits]\label{p:FinLim}
  The category~$\AsmA$ has all finite limits.
\end{proposition}

Moreover, it is clear from the above constructions in~$\AsmA$ (that
are based on the equifunctional constructions in~$\Set$)
that the forgetful functor $\Gamma:\AsmA\to\Set$ preserves the
terminal object, binary products (together with the corresponding
projections) and equalizers (together with the corresponding
inclusions).
Therefore:
\begin{proposition}[$\Gamma$ preserves all finite
    limits]\label{p:PreserveFinLim}
  The forgetful functor $\Gamma:\AsmA\to\Set$ preserves all finite
  limits in~$\AsmA$.
\end{proposition}
\subsection{Finite colimits}\label{ss:FinColim}
Let us now prove that $\AsmA$ has all finite colimits.
\begin{proposition}\label{p:Initial}
  The assembly $\mathbf{0}:=\Delta\varnothing$
  is the initial object of~$\AsmA$.
\end{proposition}

\begin{proof}
  Indeed, for each assembly~$X$, we observe that
  $$\Hom_{\AsmA}(\Delta\varnothing,X)~=~\{i_X\}\,,$$
  where $i_X$ is the inclusion of~$\varnothing$ into~$X$
  (tracked by $\top\in\S$).
\end{proof}

Now, given assemblies~$A$ and~$B$, we define the
\emph{direct sum assembly} $A+B$, letting:
\begin{itemize}
\item$\Car{A+B}~:=~\Car{A}\uplus\Car{B}
  ~=~(\{0\}\times\Car{A})\cup(\{1\}\times\Car{B})$;
\item$\Ex{A\times B}(0,a)~:=~(\Lam{xy}{x})^{\A}\iand\Ex{A}(a)$\ \
  ($\in\S$)\quad for all $a\in\Car{A}$;
\item$\Ex{A\times B}(1,b)~:=~(\Lam{xy}{y})^{\A}\iand\Ex{B}(b)$\ \
  ($\in\S$)\quad for all $b\in\Car{B}$.
\end{itemize}
We easily check that the coproduct injections 
$\sigma_1^{A,B}:\Car{A}\to\Car{A+B}$ and
$\sigma_2^{A,B}:\Car{B}\to\Car{A+B}$
are actually morphisms of types $A\to A+B$ and $B\to A+B$
(resp.)\ in~$\AsmA$, since they are tracked by the
truth values $(\Lam{x_0z}{z\,(\Lam{xy}{x})\,x_0})^{\A},
(\Lam{y_0z}{z\,(\Lam{xy}{y})\,y_0})^{\A}\in\S$
(resp.)
\begin{proposition}[Binary coproduct]\label{p:BinCoprod}
  For all assemblies $A$ and~$B$, the direct sum assembly~$A+B$
  equipped with the two morphisms $\sigma_1^{A,B}:A\to A+B$ and
  $\sigma_2^{A,B}:B\to A+B$ is the binary coproduct of~$A$ and~$B$
  in the category~$\AsmA$.
\end{proposition}

\begin{proof}
  Given an assembly~$X$ and morphisms $f:A\to X$ and $g:B\to X$,
  we want to show that there is a unique morphism $h:A+B\to X$
  such that $h\circ\sigma_1^{A,B}=f$ and $h\circ\sigma_2^{A,B}=g$:
  $$\xymatrix@C=48pt@R=48pt{
    A\ar[r]^{\ds\sigma_1^{A,B}}\ar[dr]_{\ds f}&
    A+B\ar@{-->}[d]^{\ds h}&
    B\ar[l]_{\ds\sigma_2^{A,B}}\ar[dl]^{\ds g}\\
    &X&\\
  }$$
  \smallbreak\noindent
  \emph{Existence of~$h$.}\quad
  Consider the map $h:\Car{A+B}\to\Car{X}$ defined by
  $h(0,a):=f(a)$ for all $a\in\Car{A}$ and $h(1,b)=g(b)$ for all
  $b\in\Car{B}$.
  Given trackers $s_f,s_g\in\S$ of the morphisms $f:A\to X$ and
  $g:B\to X$ (respectively), we observe that
  $$(\Lam{z}{z\,(\Lam{uv}{u\,(s_f\,v)\,(s_g\,v)})})^{\A}
  ~\ile~\bigimeet_{\mcent{c\in\Car{A+B}}}
  \bigl(\Ex{A+B}(c)\to\Ex{X}(h(c))\bigr)\,,$$
  which proves that~$h$ is a morphism of type $A+B\to X$.
  And from the definition of~$h$, it is clear that
  $h\circ\sigma_1^{A,B}=f$ and $h\circ\sigma_2^{A,B}=g$.
  \smallbreak\noindent
  \emph{Uniqueness of~$h$.}\quad
  Obvious from the equalities
  $h\circ\sigma_1^{A,B}=f$ and $h\circ\sigma_2^{A,B}=g$.
\end{proof}

Finally, given parallel arrows $f,g:A\to B$ (in the category~$\AsmA$),
we define a new assembly $\Coker(f,g)$, letting:
\begin{itemize}
\item$\Car{\Coker(f,g)}~:=~\Car{B}/{\sim}$,\quad
  where $\sim$ is the smallest equivalence relation on the set~$\Car{B}$
  such that $f(a)\sim g(a)$ for all $a\in\Car{A}$.
\item$\ds\Ex{\Coker(f,g)}(c)~:=~\bigexists_{b\in c}\Ex{B}(b)
  ~=~\bigimeet_{\mcent{v\in\A}}
  \biggl(\Bigl(\bigimeet_{b\in c}\bigl(\Ex{B}(b)\to v\bigr)\Bigr)
  \to v\biggr)$\quad for all $c\in\Car{\Coker(f,g)}$.
\item[] (Given an equivalence class
  $c\in\Car{B}/{\sim}=\Car{\Coker(f,g)}$,
  and writing $s_0:=\Ex{B}(b_0)$ for some $b_0\in c$, we observe that
  $(\Lam{z}{z\,s_0})^{\A}\ile\bigexists_{b\in c}\Ex{B}(b)$,
  which proves that $\bigexists_{b\in c}\Ex{B}(b)\in\S$.)

\end{itemize}
We also consider the epimorphism $k:B\to\Coker(f,g)$ induced by the
canonical surjection from $\Car{B}$ to $\Car{\Coker(f,g)}=\Car{B}/{\sim}$
and tracked by the truth value $(\Lam{xz}{zx})^{\A}$.

\begin{proposition}[Coequalizer]\label{p:Coequalizer}
  For all parallel arrows $f,g:A\to B$ in~$\AsmA$, the assembly
  $\Coker(f,g)$ equipped with the arrow $k:B\to\Coker(f,g)$
  is the coequalizer of~$f$ and~$g$ in~$\AsmA$.
\end{proposition}

\begin{proof}
  Given an assembly~$X$ and a morphism $h:B\to X$ such that
  $h\circ f=h\circ g$, we want to show that there is a unique
  morphism $h':\Coker(f,g)\to X$ such that $h'\circ k=h$:
  $$\xymatrix@C=40pt@R=40pt{
    A\ar@<2pt>[r]^{\ds f}\ar@<-2pt>[r]_{\ds g} &
    B\ar[r]^<<<<<<<{\ds k}\ar[dr]_{\ds h} &
    \Coker(f,g)\ar@{-->}[d]^{\ds h'} \\
    && X
  }$$
  \smallbreak\noindent
  \emph{Existence of~$h'$.}\quad
  Since $h\circ f=h\circ g$, we have $h(b)=h(b')$ for all $b,b'\in B$
  such that $b\sim b'$ (from the definition of the equivalence
  relation~$\sim$).
  Therefore, the map $h:\Car{B}\to\Car{X}$ factors through the quotient
  $\Car{B}/{\sim}$ into a map $h':\Car{B}/{\sim}\to\Car{X}$ such that
  $h'\circ k=h$.
  Given a tracker $s_h\in\S$ of the morphism $h:B\to X$, we then
  observe that
  $$(\Lam{z}{z\,s_h})^{\A}
  ~\ile~\bigimeet_{\mcent{c\in\Car{B}/{\sim}}}
  \bigl(\Ex{\Coker(f,g)}(c)\to\Ex{X}(h'(c))\bigr)\,,$$
  which proves that~$h'$ is a morphism of type $\Coker(f,g)\to X$.
  \smallbreak\noindent
  \emph{Uniqueness of~$h'$.}\quad
  Obvious from the equality $h'\circ k=h$.
\end{proof}

From Prop.~\ref{p:Initial}, \ref{p:BinCoprod} and~\ref{p:Coequalizer},
it is clear that:
\begin{proposition}[Finite colimits]\label{p:FinColim}
  The category~$\AsmA$ has all finite colimits.
\end{proposition}
\subsection{LCCC structure}\label{ss:LCCC}
\subsubsection{LCCC's}
We first recall some basic facts about \emph{locally cartesian closed categories}, commonly known as \emph{LCCC's}.

\begin{definition}
  Let $\mathbb{B}$ be a category.
  \begin{enumeratenumeric}
    \item Fix an object $X \in \mathbb{B}$. The {\tmem{slice category
    $\slice{\mathbb{B}}{X}$}} {\tmem{of $\mathbb{B}$ over $X$}} (or simply
    {\tmem{slice}} when everything is understood) has morphisms of
    $\mathbb{B}$ with codomain $X$ as objects. A morphism $$u : \arobj{s}{V}{X}
    \rightarrow \arobj{t}{W}{X}$$ is a commuting triangle
$$
\xymatrix @-.5pc {
V \ar[rr]^u \ar[dr]_<(.3)s & & W \ar[dl]^<(.3)t \\
& X &
}
$$

    Composition is given by horizontal pasting of such triangles.

    \item The \tmem{arrow category}
    $\arcat{\mathbb{B}}$
    of $\mathbb{B}$ has
    morphisms of $\mathbb{B}$ as objects.
    A morphism $$(g, f) : \arobj{t}{W}{X} \rightarrow \arobj{p}{A}{Y}$$ in $\arcat{\mathbb{B}}$ is a commuting square
    $$
    \xymatrix @-.1pc {
    W \ar[r]^g \ar[d]_t & A \ar[d]^p \\
    X \ar[r]_f & Y
    }
    $$

    in $\mathbb{B}$. Composition is given by horizontal pasting of
    such squares.
  \end{enumeratenumeric}
\end{definition}

\begin{remark}
  Let $\mathbb{B}$ be a category.
  \begin{enumeratenumeric}
    \item There is the obvious functor $\tmop{cod} : \arcat{\mathbb{B}}
    \rightarrow \mathbb{B}$ sending an object $\smallarobj{t}{W}{X}$ on it's
    codomain $X$ and a morphism $$(g, f) : \arobj{t}{W}{X} \rightarrow
    \arobj{p}{A}{Y}$$ on it's {\tmem{codomain morphism}} $f$.

    \item If $\mathbb{B}$ has pullbacks, $\tmop{cod}$ is a Grothendieck
    fibration with fibers the slice categories
    $\slice{\mathbb{B}}{X}$ for all $X \in \mathbb{C}$. Reindexing an object
    $\smallarobj{p}{A}{Y} \in \slice{\mathbb{B}}{Y}$ along a morphism $f : X
    \rightarrow Y$ in the base $\mathbb{B}$ is given by pullback
    $$
    \xymatrix{
        p^*A \ar[r] \ar[d]_{f^*p} \pullbackcorner & A \ar[d]^p \\
        X \ar[r]_f & Y
    }
    $$

    A cleavage corresponds to a choice of pullbacks.
  \end{enumeratenumeric}
\end{remark}

\begin{definition}
  Let $\mathbb{B}$ be a category with finite limits.
  \begin{enumeratenumeric}
    \item $\mathbb{B}$ is cartesian closed (or simply CCC) if for all objects
    $A \in \mathbb{B}$ the functor $(-) \times A$ has a right adjoint
    $A^{(-)}$. This right adjoint is commonly known as {\tmem{exponentiation}}.

    \item $\mathbb{B}$ with finite limits is locally cartesian closed (or
    simply LCCC) if each slice $\mathbb{B}/ X$ is a CCC and reindextion
    functors in $\tmop{cod} : \mathbb{B}^{\rightarrow} \rightarrow \mathbb{B}$
    are cartesian closed (preserve binary products and exponentiation).
  \end{enumeratenumeric}
\end{definition}

\begin{theorem}
  \label{th:lccc}\tmtextbf{}Let $\mathbb{B}$ be a category with finite limits.
  The following are equivalent
  \begin{enumerateroman}
    \item $\mathbb{B}$ is an LCCC;

    \item The functor $f^{\ast} : \mathbb{B}/ Y \rightarrow \mathbb{B}/ X$ has
    a right adjoint $\depr{f} : \mathbb{B}/ X \rightarrow \mathbb{B}/ Y$
    for all morphisms $f : X \rightarrow Y$.
  \end{enumerateroman}
\end{theorem}

Theorem \ref{th:lccc} (perhaps not the most general form of the statement) is part of the lore, see for instance \cite{freyd1972aspects}, \cite{seely1984locally} or \cite{awodey2010category}.

\subsubsection{Sets}

Assume a set-theoretical function $f: X \to Y$.
We shall now carry out the proof that the functor
$f^{\ast} :\slice{\ens}{Y} \to \slice{\ens}{X} $ has
a right adjoint. This obviously implies that $\ens$ is an LCCC by virtue of Theorem \ref{th:lccc}, which is largely irrelevant in itself since it is a well-known fact (true for any topos as it happens). The reason for the burden is that we shall need the constructions subsequently, in order to prove that $\AsmA$ is an LCCC. We do not claim any originality here, yet were unable to locate a precise source for the technique we are using.

\begin{definition}
  Let $f : X \rightarrow Y$ be a map and $y \in Y$ an element we shall call
  {\tmem{basepoint}}. An
  {\tmem{$f$-section}} of an object $\smallarobj{t}{W}{X} \in \slice{\ens}{X}$ {\tmem{over the basepoint $y \in Y$}} is a map
  $s : f^{- 1} (y) \rightarrow W$ such that $t \circ s = \tmop{id}$.
\end{definition}

Given $f : X \rightarrow Y$ and a basepoint $y \in Y$, an $f$-section $s$ of
$\smallarobj{t}{W}{X}$ makes the diagram
$$
\xymatrix{
& W \ar[d]^t \\
f^{-1}(y) \ar[ur]^s \ar@{^{(}->}[r] & X}
$$

commute. An $f$-section of $t$ over some basepoint does not need to
be defined on {\tmem{all}} $X$, it is a {\tmem{partial map}} from $X$ to $W$.

\begin{notation*}
  Given a map $f : X \rightarrow Y$ and an object $\smallarobj{t}{W}{X} \in
  \slice{\ens}{X}$, we shall write
  \begin{eqnarray*}
    \ppsect{f}{t} & \assign & \left\{ (y, s) \:|\: y \in Y, s \in \ens (f^{- 1}
    (y), W), t \circ s = \tmop{id} \right\}
  \end{eqnarray*}
  for the set of all $f$-sections of $\smallarobj{t}{W}{X}$ taken over all
  basepoints and $\tmop{bp} : \ppsect{f}{t} \rightarrow Y$ for the first
  projection.
\end{notation*}

The fancy naming of the above first projection suggests that the latter sends
a partial $f$-section on it's basepoint.

\begin{proposition}
  \label{prop:fun}The map of classes
  \begin{eqnarray*}
    \depr{f} : \left( \slice{\ens}{X} \right)_0 & \longrightarrow & \left(
    \slice{\ens}{Y} \right)_0\\
    \arobj{t}{W}{X} & \mapsto &  \arobj{{\tmop{bp}}}{{\ppsect{f}{t}}}{Y}
  \end{eqnarray*}
  extends to a functor $\depr{f} : \slice{\ens}{X} \rightarrow
  \slice{\ens}{Y}$ with action on morphisms given by
  \begin{eqnarray*}
    \depr{f} (u) : \arobj{{\tmop{bp}}} {{\ppsect{f}{t}}}{Y} & \longrightarrow &
    \arobj{{\tmop{bp}}}{{\ppsect{f}{t'}}}{Y} \\
    (y, s) & \mapsto & (y, u \circ s)
  \end{eqnarray*}
\end{proposition}

\begin{proof}
  Let $$u : \arobj{t}{W}{X} \rightarrow \arobj{{t'}}{{W'}}{X}$$ be a morphism in
  $\slice{\ens}{X}$ and $s : f^{- 1} (y) \rightarrow W$ be an $f$-section of
  $t$ over the basepoint $y \in Y$. The map $u \circ s : f^{- 1} (y)
  \rightarrow W'$ is an $f$-section of $t'$ above the same basepoint $y \in Y$
  since the diagram
$$
\xymatrix{
& W \ar[rr]^u \ar[dr]_t & & W' \ar[dl]^{t'} \\
f^{-1}(y) \ar@{^{(}->}[rr] \ar[ur]^s & & X &
}
$$

  commutes. We thus have a map
  \begin{eqnarray*}
    \depr{f} (u) : \arobj{{\tmop{bp}}}{{\ppsect{f}{t}}}{Y} & \longrightarrow &
    \arobj{{\tmop{bp}}}{{\ppsect{f}{t'}}}{Y} \\
    (y, s) & \mapsto & (y, u \circ s)
  \end{eqnarray*}
  It is immediate that the above are the data of a functor $\depr{f} :
  \slice{\ens}{X} \rightarrow \slice{\ens}{Y}$.
\end{proof}

\begin{lemma}
  \label{lem:part-sec}Assume $f : X \rightarrow Y$, a basepoint $y \in Y$ and
  $v : f^{- 1} (y) \rightarrow p^{- 1} (y)$. The canonical map
  \begin{eqnarray*}
    \langle \tmop{id}, v \rangle : f^{- 1} (y) & \longrightarrow &
    \pfib{X}{Y}{A}
  \end{eqnarray*}
  is an $f$-section of $p_1 : \pfib{X}{Y}{A} \rightarrow X$ over $y$.
\end{lemma}

\begin{proof}
  Assume $x \in f^{- 1} (y)$. We have
  \begin{eqnarray*}
    (f \circ \tmop{id}) (x) & = & y\\
    & = & (p \circ v) (x)
  \end{eqnarray*}
  hence the outer diagram in
$$
\xymatrix{
f^{-1}(y) \ar@{=}[dd] \ar[dr]^{\langle id,v \rangle} \ar[rr]^v && p^{-1}(y) \ar@{^{(}->}[d]\\
& X \times_Y A \ar[d]_{p_1} \ar[r]^{p_2} \pullbackcorner & A \ar[d]^p \\
f^{-1}(y) \ar@{^{(}->}[r] & X \ar[r]_f & Y
}
$$
  commutes, so $\langle \tmop{id}, v \rangle$ is the canonical map.
\end{proof}

\begin{proposition}\label{prop:adj-set}
  \label{prop:adj}The functor $\depr{f} : \slice{\ens}{X} \rightarrow
  \slice{\ens}{Y}$ is right adjoint to the reindexation functor $f^{\ast} :
  \slice{\ens}{Y} \rightarrow \slice{\ens}{X}$.
\end{proposition}

\begin{proof}
  Assume $\smallarobj{p}{A}{Y} \in \slice{\ens}{Y}$ and $a \in A$. Let
  \begin{eqnarray*}
    k_a : f^{- 1} (p (a)) & \longrightarrow & p^{- 1} (p (a))\\
    x & \mapsto & a
  \end{eqnarray*}
  be the constant map. The map $\langle \tmop{id}, k_a \rangle : f^{- 1} (p
  (a)) \rightarrow \pfib{X}{Y}{A}$ is an $f$-section over $p (a)$ (c.f. Lemma
  \ref{lem:part-sec}) so the map
  \begin{eqnarray*}
    \eta_p : A & \longrightarrow & \ppsect{f}{p}\\
    a & \mapsto & (p (a), \langle \tmop{id}, k_a \rangle)
  \end{eqnarray*}
  is well-defined. Moreover, it is (trivially) a morphism $$\eta_p :
  \arobj{p}{A}{Y} \rightarrow \arobj{{\tmop{bp}}}{{\ppsect{f}{p}}}{Y}$$ in
  {\slice{{\ens}}{Y}}. Assume
  $\smallarobj{p}{A}{Y}
  \in \slice{\ens}{Y}$.
  We have
  \begin{eqnarray*}
    \left( \depr{f} \circ f^{\ast} \right)  \arobj{p}{A}{Y} & = &
    \depr{f} \arobj{{p_1}}{{X \times_Y A}}{X}\\
    & = & \arobj{{\tmop{bp}}}{{\ppsect{f}{{p_1}}}}{Y}
  \end{eqnarray*}
  and claim that the morphism $\eta_p$ is universal from $\smallarobj{p}{A}{Y}$
  to $\depr{f}$.

  Assume $\smallarobj{q}{B}{X} \in \slice{\ens}{X}$ and
  \begin{eqnarray*}
    w : \arobj{p}{A}{Y} & \longrightarrow & \depr{f} \arobj{q}{B}{X} =
    \arobj{{\tmop{bp}}}{{\ppsect{f}{q}}}{Y}
  \end{eqnarray*}
  a morphism in $\slice{\ens}{Y}$. We have in particular
  \begin{eqnarray*}
    w (a) & = & (p (a), s_{w, a})
  \end{eqnarray*}
  for some $f$-section $s_{w, a}$ of $q$ over the basepoint $p (a)$. Now $(x,
  a) \in X \times_Y A$ implies that $x \in f^{- 1} (p (a))$. Let
  \begin{eqnarray*}
    \bar{w} : X \times_Y A & \longrightarrow & B\\
    (x, a) & \mapsto & s_{w, a} (x)
  \end{eqnarray*}
  But $s_{w, a}$ is an $f$-section of $\smallarobj{q}{B}{X}$, so $q \circ s_{w, a}
  = \tmop{id}$ hence
  \begin{eqnarray*}
    \bar{w} : f^{\ast} \arobj{p}{A}{Y} = \arobj{{p_1}}{{X \times_Y A}}{X} &
    \longrightarrow & \arobj{q}{B}{X}
  \end{eqnarray*}
  is a morphism in $\slice{\ens}{X}$. We have the map
  \begin{eqnarray*}
    \depr{f} (\bar{w}) : \ppsect{f}{{p_1}} & \longrightarrow & \ppsect{f}{q}\\
    (y, s) & \mapsto & (y, \bar{w} \circ s)
  \end{eqnarray*}
  which is a morphism $$\depr{f} (\bar{w}) : \arobj{{\tmop{bp}}}{{\ppsect{f}{{p_1}}}}{Y} \rightarrow
  \arobj{{\tmop{bp}}}{{\ppsect{f}{q}}}{Y}$$ in
  $\slice{\ens}{Y}$ (c.f. Proposition \ref{prop:fun}) so in particular
  \begin{eqnarray*}
    \left( \depr{f} (\bar{w}) \circ \eta_p \right) (a) & = & \depr{f}
    (\bar{w}) (p (a), \langle \tmop{id}, k_a \rangle)\\
    & = & (p (a), \bar{w} \circ \langle \tmop{id}, k_a \rangle)
  \end{eqnarray*}
  We further have
  \begin{eqnarray*}
    (\bar{w} \circ \langle \tmop{id}, k_a \rangle) (x) & = & \bar{w} (x, k_a
    (x))\\
    & = & \bar{w} (x, a)\\
    & = & s_{w, a} (x)
  \end{eqnarray*}
  for all $x \in f^{- 1} (p (a))$, hence $w = \depr{f} (\bar{w}) \circ
  \eta_p$. In other words, the diagram
$$
\xymatrix{
{\begin{bsmallmatrix}A \\ \;\; \downarrow p \\Y \end{bsmallmatrix}}
\ar[rr]^(.35){\eta_p} \ar[drr]_w
&&
\left ( \depr{f} \circ f^\ast \right )
{\begin{bsmallmatrix}A \\ \;\; \downarrow p \\Y \end{bsmallmatrix}}
\ar[d]^{\depr{f} (\bar{w})}
\\
&& \depr{f}
{\begin{bsmallmatrix}B \\ \;\; \downarrow q \\ X \end{bsmallmatrix}}
}
$$
  in {\slice{{\ens}}{Y}} commutes. It is straightforward to verify that  $\depr{f} (\bar{w})$ is the unique map with this property.

\end{proof}

\subsubsection{Existence and Tracking}

Assemblies are basically sets with bells and whistles while the functor
$\Gamma :
\AsmA \rightarrow \ens$ creates and
preserves finite limits (c.f. section \ref{s:AsmA}). Our strategy to prove that $\AsmA$ is an LCCC is thus to
add realizers to the set-theoretical constructions above. In order to retrieve some of the
latter we need to trace the relevant proofs. We first introduce some human
readable notation for $\lambda$-calculus as we shall systematically exploit the fact that any $\lambda$-term with parameters in $\S$
is in $\S$ \cite{Miq20}.

\begin{notation*}
  We shall systematically use the following $\lambda$-calculus macros
  \begin{eqnarray*}
    \upsilon \odot \sigma & \assign &
    \lam{x}{\upsilon (\sigma x)}
    \qquad x
    \nin \tmop{FV} (\upsilon) \cup \tmop{FV} (\sigma)
    \\
    \lpair{\theta}{\sigma} & \assign &
    \lam{s}{s \theta \sigma}
    \hspace{3em}
    s \nin \tmop{FV} (\theta) \cup \tmop{FV} (\sigma)
    \\
    \lpione & \assign & \lam{z}{z \left( \lam{x \nocomma
    y}{x} \right)}
    \\
    \lpitwo & \assign & \lam{z}{z \left( \lam{x \nocomma y}{y}
    \right)}
  \end{eqnarray*}
\end{notation*}

\begin{remark}
  \label{rem:pull}Recall that $\AsmA$
  has finite limits (c.f. Section \ref{ss:FinLim}). Consider the pullback square
  $$
  \xymatrix{
      {X \times_Y A} \ar[r]^{p_2} \ar[d]_{p_1} \pullbackcorner & A \ar[d]^p \\
      X \ar[r]_f & Y
  }
  $$

  We have
  \begin{itemize}
    \item $\car{\pfib{X}{Y}{A}} = \underline{X} \times_{\underline{Y}}
    \underline{A}$ with $\exi{X \times_Y A} (x, a) = \exi{X} (x) \sqcap
    \exi{A} (a)$;

    \item the first projection $p_1 : \underline{X} \times_{\underline{Y}}
    \underline{A} \rightarrow \underline{X}$ tracked by
    $\intrpa{\lpione}$; 
    \item the second projection $p_2 : \underline{X} \times_{\underline{Y}}
    \underline{A} \rightarrow \underline{A}$ tracked by
    $\intrpa{\lpitwo}$. 
  \end{itemize}
  Assume $u : Z \rightarrow X$ tracked by $\varphi \in \S$ and $v : Z
  \rightarrow A$ tracked by $\psi \in \S$ such that $p \circ v = f \circ u$.
  The canonical morphism $\langle u, v \rangle : Z \rightarrow X \times_Y A$
  is given by
  \begin{itemize}
    \item $\underline{\langle u, v \rangle} = \left\langle \underline{u},
    \underline{v} \right\rangle$;

    \item the tracker $\intrpa{\lambda x. \lpair{\varphi x}{\psi x}} =
    \intrpa{\lam{x \nocomma}{z (\varphi x) (\psi x)}}$.
  \end{itemize}
  The fiber $f^{- 1} (y)$ of $y \in Y$ can be seen as a particular case, its
  existence predicate being $$\exi{f^{- 1} (y)} (x) = \exi{X} (x) \sqcap \exi{Y}
  (y)$$
\end{remark}

\begin{definition}
  Let $f : X \rightarrow Y$ be a morphism \ in
  $\AsmA$, $\smallarobj{t}{W}{X} \in
  \slice{\AsmA}{X}$ and a basepoint $y
  \in Y$. An $f$-section of $t$ is a morphism $s \in
  \AsmA (f^{- 1} (y), W)$ such that $t
  \circ s = \tmop{id}$.
\end{definition}

An $f$-section in $\AsmA$ is thus just an
$f$-section admitting a tracker, which is rather unsurprising.

\begin{notation*}
  Given a morphism $f : X \rightarrow Y$ in
  $\AsmA$ and an object $\smallarobj{t}{W}{X}
  \in \slice{\AsmA}{X}$, we shall overload notation and write
  \begin{eqnarray*}
    \apsect{f}{t} & \assign & \{ (y, s) |y \in Y, s \in
    \AsmA (f^{- 1} (y), W), t \circ s =
    \tmop{id} \}
  \end{eqnarray*}
  for the set of all $f$-sections of $\smallarobj{t}{W}{X}$ taken over all
  basepoints and ${\tmop{bp}} : \apsect{f}{t} \rightarrow Y$ for the first
  projection.
\end{notation*}

\begin{lemma}
  \label{lem:psect-asm}Let $\smallarobj{t}{W}{X} \in
  \slice{\AsmA}{X}$.
  \begin{enumeratenumeric}
    \item $\apsect{f}{t}$ is an assembly with existence predicate given by
    \begin{eqnarray*}
      \Ex{\apsect{f}{t}} \assign \Ex{Y} \iand
      \bigimeet_{x \in f^{- 1} (y)} \left( \exi{X} (x) \rightarrow
      \exi{W} (s (x)) \right)
    \end{eqnarray*}
    \item $\intrpa{\lpione}$
    is a
    tracker for the basepoint map $\tmop{bp} : \apsect{f}{t} \rightarrow Y$.
  \end{enumeratenumeric}
\end{lemma}

\begin{proof}
  Assume $(y, s) \in \apsect{f}{t}$ with $s$ tracked by $\sigma$. We thus
  have
  \begin{eqnarray*}
    \sigma & \ile & \bigimeet_{x \in f^{- 1} (y)} \left( \exi{X}
    (x) \rightarrow \exi{W} (s (x)) \right)
  \end{eqnarray*}
  by hypothesis, hence
  \begin{eqnarray*}
    \intrpa{\lpair{\exi{Y} (y)}{\sigma}} & \ile & \exi{Y} (y) \sqcap
    \bigimeet_{x \in f^{- 1} (y)} \left( \exi{X} (x) \rightarrow \exi{W}
    (s (x)) \right)
  \end{eqnarray*}
  where
  \begin{eqnarray*}
    \intrpa{\lpair{\exi{Y} (y)}{\sigma}} & = & \intrpa{\lambda z.z\: \exi{Y} (y) \:\sigma}
  \end{eqnarray*}
  hence
  \begin{eqnarray*}
    \exi{Y} (y) \sqcap \bigimeet_{x \in f^{- 1} (y)} \left( \exi{X} (x)
    \rightarrow \exi{W} (s (x)) \right) & \in & \S
  \end{eqnarray*}
The second item is obvious (c.f. Section \ref{ss:FinLim}).
\end{proof}

\begin{proposition}
  \label{prop:fun-asm}
  The map of classes
  \begin{eqnarray*}
    \depr{f} : \left(
    \slice{\AsmA}{X} \right)_0 &
    \longrightarrow & \left(
    \slice{\AsmA}{Y} \right)_0
     \\
    \arobj{t}{W}{X} & \mapsto &  \arobj{{\tmop{bp}}}{{\apsect{f}{t}}}{Y}
  \end{eqnarray*}
  extends to a functor $\depr{f} :
  \slice{\AsmA}{X} \rightarrow
  \slice{\AsmA}{Y}$ with action on
  morphisms given by
  \begin{eqnarray*}
    \depr{f} (u) : \arobj{{\tmop{bp}}}{{\apsect{f}{t}}}{Y} & \longrightarrow &
    \arobj{{\tmop{bp}}}{{\apsect{f}{t'}}}{Y} \\
    (y, s) & \mapsto & (y, u \circ s)
  \end{eqnarray*}
\end{proposition}

Notice that we overload notation here as well.

\begin{proof}
$\smallarobj{{\tmop{bp}}}{{\apsect{f}{t}}}{Y} \in \slice{\AsmA}{X}$
while
$\smallarobj{{\tmop{bp}}}{{\apsect{f}{t'}}}{Y} \in \slice{\AsmA}{Y}$
(c.f. Lemma \ref{lem:psect-asm}). Moreover, $\depr{f} (u)$ is the action of a functor on the carriers (c.f. Proposition \ref{prop:fun-asm}). Assume $\upsilon$ tracks $u$. We thus have
  \begin{eqnarray*}
    \intrpa{\lambda \lpair{\theta}{\sigma} . \lpair{\theta}{\upsilon \odot \sigma}} &
    \ile & \exi{Y} (y) \sqcap \bigimeet_{x \in f^{- 1} (y)}
    \left( \exi{X} (x) \rightarrow \exi{W} (s (x)) \right)\\
    &  & \quad \to \;\; \exi{Y} (y) \sqcap \bigimeet_{x \in f^{- 1} (y)} \left( \exi{X}
    (x) \rightarrow \exi{W} ((u \circ s) (x)) \right)
  \end{eqnarray*}
  so
  \begin{eqnarray*}
    \intrpa{\lambda \lpair{\theta}{\sigma} . \lpair{\theta}{\upsilon \odot \sigma}}
    & =&
    \intrpa{\lambda w. \lambda s.s \left( \left( \lam{z}{z \left( \lam{x
    \nocomma y}{x} \right)} \right) (w) \right) \left( \lambda x. \upsilon
    \left( \left( \left( \lam{z}{z \left( \lam{x \nocomma y}{y} \right)}
    \right) (w) \right) x \right) \right)}
  \end{eqnarray*}
  is a tracker of $\depr{f} (u)$.
\end{proof}

{\skol{\label{skol:part-sec}Assume $f : X \rightarrow Y$ in
$\AsmA$, a basepoint $y \in
Y$ and $v : f^{- 1} (y) \rightarrow p^{- 1} (y)$ in
$\AsmA$. The pairing map
\begin{eqnarray*}
  \langle \tmop{id}, v \rangle : f^{- 1} (y) & \longrightarrow & X \times_Y Y
\end{eqnarray*}
is an $f$-section of $p_1 : X \times_Y A \rightarrow X$ over $y$ in
$\AsmA$.}}

\begin{proof}
  Consider the diagram in the proof of Lemma \ref{lem:part-sec}, seen as a
  diagram in $\AsmA$. All the reasoning
  at the level of the carriers is still valid, in particular that $\langle
  \tmop{id}, v \rangle $ is an $f$-section of the carrier $p_1 : \underline{X}
  \times_{\underline{Y}} \underline{A} \rightarrow \underline{X}$. But now all
  the objects are equipped with existence predicates while all the maps are
  tracked (c.f. Remark \ref{rem:pull}).
\end{proof}

{\skol{\label{skol:adj}The functor $\depr{f} :
\slice{\AsmA}{X} \rightarrow
\slice{\AsmA}{Y}$ is right adjoint to
the reindexation functor $f^{\ast} :
\slice{\AsmA}{Y} \rightarrow
\slice{\AsmA}{X}$.}}

\begin{proof}
  We trace the proof of Proposition \ref{prop:adj} and equip the relevant maps
  with trackers.
  Assume $\smallarobj{p}{A}{Y} \in
  \slice{\AsmA}{Y}$ with $p$ tracked by
  $\varpi$. Assume $a \in A$. The map
  \begin{eqnarray*}
    \eta_p : A & \longrightarrow & \apsect{f}{p}\\
    a & \mapsto & (p (a), \langle \tmop{id}, k_a \rangle)
  \end{eqnarray*}
  is well-defined with respect to $\AsmA$ (c.f. Scholium
  \ref{skol:part-sec}). We further have
  \begin{eqnarray*}
    \intrpa{\lambda s. \lpair{\varpi \nocomma s}{\lpair{\lambda x.x}{\lambda x.s}}} &
    \ile &  \left( \exi{A} (a) \rightarrow
    \exi{{\apsect{f}{p}}} (p (a), \langle \tmop{id}, k_a \rangle)
    \right)
  \end{eqnarray*}
  so
  \begin{eqnarray*}
    \intrpa{\lambda s. \lpair{\varpi \nocomma s}{\lpair{\lambda x.x}{\lambda x.s}}} & =
    & \intrpa{\lam{s}{\lam{t}{t \nocomma (\varpi \nocomma s) \left( \lam{x \nocomma
    z}{z \nocomma x \nocomma s} \right)}}}
  \end{eqnarray*}
  is a tracker for $\eta_p$, hence $\eta_p$ is a morphism in $\AsmA$. It follows that
  \begin{eqnarray*}
    \eta_p : \arobj{p}{A}{Y} & \to &
    \arobj{{\tmop{bp}}}{{\apsect{f}{p}}}{Y}
  \end{eqnarray*}
is a morphism in $\slice{\AsmA}{Y}$, this since $\tmop{bp}$ admits a tracker (c.f. Lemma \ref{lem:psect-asm}.2) while the relevant triangle of carriers commutes (c.f. Proposition \ref{prop:adj-set}).

  Let $\smallarobj{q}{B}{X} \in
  \slice{\AsmA}{X}$ with $q$ tracked by
  $\chi$. Let
  \begin{eqnarray*}
    w : \arobj{p}{A}{Y} & \longrightarrow & \depr{f}  \arobj{q}{B}{X} =
    \arobj{{\tmop{bp}}}{{\apsect{f}{q}}}{Y}
  \end{eqnarray*}
  be a morphism in $\slice{\AsmA}{Y}$
  where
  \begin{eqnarray*}
    w (a) & = & (p (a), s_{w, a})
  \end{eqnarray*}
  for some $f$-section $s_{w, a}$ of $q$ over the basepoint $p (a)$. Assume $\xi$ tracks $w$.

  Recall
  that the canonical map on the carriers with respect to $\eta_p$ is
  $\depr{f}(\bar{w})$
with
  \begin{eqnarray*}
    \bar{w} : X \times_Y A & \longrightarrow & B\\
    (x, a) & \mapsto & s_{w, a} (x)
  \end{eqnarray*}
  All we need to complete the proof is thus a tracker for $\bar{w}$, which is of type $$\bigimeet_{(x,a) \in \pfib{X}{Y}{A}} \exi{\pfib{X}{Y}{A}} (x, a)
  \to \exi{B} (\bar{w} (x, a))$$
    Suppose $\chi$ realises $\exi{X} (x)$
  and $\alpha$ realises $\exi{A} (a)$. Then $\xi \alpha$ realises
  $\mathbf{E}_{\apsect{f}{q}} (p (a), s_{w, a})$, hence $\lpitwo
  (\xi \alpha)$ tracks $s_{w, a}$ by definition of $\mathbf{E}_{\apsect{f}{q}}$.
  We thus have
  \begin{eqnarray*}
    \intrpa{\lambda \lpair{\chi}{\alpha} .\lpitwo (\xi \alpha) \chi} &
    \ile & \exi{\pfib{X}{Y}{A}} (x, a) \rightarrow \exi{B} (s_{w, a} (x))
    \\
    & = & \exi{\pfib{X}{Y}{A}} (x, a) \rightarrow \exi{B} (\bar{w}(x, a))
  \end{eqnarray*}
  so
  \begin{eqnarray*}
    \intrpa{\lambda \lpair{\chi}{\alpha} .\lpitwo (\xi \alpha) \chi} & = &
    \intrpa{\lambda v.\lpitwo \left( \xi \left( \lpitwo v \right) \right)
    \left( \lpione v \right)}\\
    & = & \intrpa{\lambda v. \left( \lam{z}{z \left( \lam{x \nocomma y}{y} \right)}
    \right) \left( \xi \left( v \left( \lam{z}{z \left( \lam{x \nocomma y}{y}
    \right)} \right) \right) \right) \left( \left( \lam{z}{z \left( \lam{x
    \nocomma y}{x} \right)} \right) v \right)}
  \end{eqnarray*}
  is a tracker for $\bar{w}$.
\end{proof}

\begin{theorem}
  \label{th:asm-lccc}
$\AsmA$ is an LCCC.
\end{theorem}
\begin{proof}
Scholium \ref{skol:adj} and Theorem \ref{th:lccc}.
\end{proof}
\subsection{Strong subobject classifier}\label{ss:SSOC}
Let us consider the assembly 
$\SOC:=\Delta{\mathcal{P}(1)}$ 
and the morphism $t:\mathbf{1}\to\SOC$ 
defined by $t(*):=1$ and tracked by $\top$.

As a tool to show that $\SOC$ (equipped
with $t$) is a strong
suboject classifier, we note that every morphism
$f:A\to B$ (tracked by $s_f$)
of $\AsmA$ can be factorized into
an epimorphism $\hat{f}:A\to\Im(f)$
followed by a monomorphism $m_f : \Im(f)\to B$ 
where $\Im(f)$ is the
assembly defined by :
\begin{itemize}
\item $\Car{\Im(f)}~:=~\{ f(a) \, | \, a \in B\}$
\item $\ds\Ex{\Im(f)}(b)~:=~\ds\Ex{B}(b) 
  $\quad
  for all $b\in \Im(f)$
\end{itemize}
while $\hat{f}$ and $m_f$ are defined by 
$\hat{f}(a):=f(a)$ for all $a \in A$
and $m_f(b):=b$ for all $b \in \Im(f)$, that
are respectively tracked by $s_f$ and
$(\Lam{x}{x})^{\A}$.
In particular, if $f$ is an extremal monomorphism
\cite{Kel69} then $\hat{f}$ is an isomorphism.  

\begin{remark}
Although the carrier of $\Im(f)$ is constructed
as the set theoretic image of $f$, the monomorphism
$m_f : \Im(f)\to B$ is in general 
not the image \cite{Mit69} of the morphism $f$.
As for the epimorphism $\hat{f}$, it is in
general not regular.
\end{remark}

\begin{proposition}[Strong Suboject Classifier]
\label{p:SSOC}
  The assembly $\SOC$ equipped with the arrow
  $t:\mathbf{1}\to\SOC$ is a strong subobject
  classifier in the category $\AsmA$.
\end{proposition}
\begin{proof}
Let $f:A\to B$ be a strong monomorphism.
As a strong monomorphism, $f$ is an extremal
monomorphism
and therefore $\hat{f}: A\to\Im(f)$
is an isomorphim \cite{Kel69}. It follows that
$f$ and $m_f:\Im(f)\to B$ represent the same
subobject.

Finally, the classifying map of $m_f$ 
is $\chi_f : B\to\Omega$ defined by 
$\chi_f(b) := \{ x \in 1 \, | \, b \in \Car{\Im(f)}\}$ and tracked by $\top$.
\end{proof}
\subsection{The natural numbers object}\label{ss:NNO}
Let us now consider the assembly $\N$ defined by:
\begin{itemize}
\item $\Car{\NN}~:=~\N$
\item $\ds\Ex{\NN}(n)~:=~
  \bigimeet_{\mcent{a\in\A^{\N}}}
  \bigl(a_0\to\bigimeet_{\mcent{p\in\N}}(a_p\to a_{p+1})\to a_n\bigr)$\quad
  for all $n\in\N$.
\item[](By induction on $n\in\N$, we check that
  $(\Lam{xf}{f^nx})^{\A}\ile\Ex{\NN}(n)$, hence
  $\Ex{\NN}(n)\in\S$.)
\end{itemize}
We also consider the morphisms $z:\mathbf{1}\to\NN$ and
$s:\NN\to\NN$ defined by $z(*):=0$ and $s(n):=n+1$ for all $n\in\N$,
that are respectively tracked by
$(\Lam{zxf}{x})^{\A},(\Lam{nxf}{f\,(n\,x\,f)})^{\A}\in\S$.

\begin{proposition}[Natural Numbers Object]\label{p:NNO}
  The assembly $\NN$ equipped with the two arrows
  $z:\mathbf{1}\to\NN$ and $s:\NN\to\NN$ is a natural numbers
  object in the category $\AsmA$.
\end{proposition}

\begin{proof}
  Given an assembly $A$ and two morphisms $q:\mathbf{1}\to X$
  and $f:X\to X$, we want to show that there is a unique morphism
  $u:\NN\to X$ such that the following diagram commutes:
  $$\xymatrix@R=40pt@C=40pt{
    \mathbf{1}\ar[r]^{\ds z}\ar[dr]_{\ds q}
    &\NN\ar[r]^{\ds s}\ar@{-->}[d]^{\ds u}&\NN\ar@{-->}[d]^{\ds u} \\
    &X\ar[r]_{\ds f}&X\\
  }$$
  \smallbreak\noindent
  \emph{Existence of~$u$.}\quad
  Consider the map $u:\N\to\Car{X}$ defined by $u(0)=q(0)$
  and $u(n+1)=f(u(n))$ for all $n\in\N$.
  Given trackers $s_q,s_f\in\S$ of the morphisms $q:\mathbf{1}\to X$
  and $f:X\to X$ (resp.), we observe that\ \
  $s_q\,\top\ile\Ex{X}(u(0))$\ \ and\ \
  $s_f\ile\bigimeet_{p\in\N}
  \bigl(\Ex{X}(u(p))\to\Ex{X}(u(p+1))\bigr)$\ \
  (from the definition of the map~$u$), from which we deduce that
  $$(\Lam{m}{m\,(s_q\,\top)\,s_f})^{\A}~\ile~
  \bigimeet_{\mcent{n\in\N}}\bigl(\Ex{\NN}(n)\to\Ex{X}(u(n))\bigr)$$
  using techniques of semantic typing.
  Therefore~$u$ is a morphism of type $\NN\to X$.
  \smallbreak\noindent
  \emph{Uniqueness of~$u$.}\quad
  Obvious from the equalities $u\circ z=q$ and
  $u\circ s=f\circ u$.
\end{proof}

\section{The particular case of forcing}
\label{s:Forcing}

\subsection{Implicative algebras and forcing}
As mentioned in~\Sect~\ref{s:ImpAlg}, complete Heyting (or Boolean)
algebras $(H,\ile)$ are particular cases of implicative structures,
namely: the implicative structures whose implication `$\to$' is
derived from the ordering `$\ile$' using Heyting's adjunction:
$$c\ile(a\to b)\qquad\text{iff}\qquad(c\imeet a)\ile b
\eqno(\text{for all}~a,b,c\in H)$$
Such implicative structures are immediately turned into implicative
algebras, just by endowing them with the trivial separator
$\S:=\{\top\}$.
The corresponding tripos
$\P:\Set^{\op}\to\HA$
is then (naturally)
isomorphic to the forcing tripos
$\Hom_{\Set}(\Anon,\,H):\Set^{\op}\to\HA$.

However, there are many cases of implicative algebras~$\A$ whose
associated tripos turn out to be isomorphic to a forcing tripos
(i.e.\ of the form $\Hom_{\Set}(\Anon,\,H)$ for some complete Heyting
algebra~$H$) although the complete lattice underlying~$\A$ is not a
Heyting algebra.
In~\cite{Miq20}, this situation is captured by the following
theorem (Theorem 4.13,~\Sect~4.5):
\begin{theorem}[Characterizing forcing triposes]%
  \label{th:CharacForTrip}
  For each implicative algebra $\A$ (writing~$\S$ its separator),
  the following are equivalent:
  \begin{enumerate}
  \item The tripos
  induced by~$\A$ is isomorphic to a forcing tripos.
  \item The separator $\S\subseteq\A$ is a principal filter of~$\A$.
  \item The separator $\S\subseteq\A$ is finitely generated and
    contains the non-deterministic choice operator
    $\Fork^{\A}:=\bigimeet_{a,b\in\A}(a\to b\to a\imeet b)
    =(\Lam{xy}{x})^{\A}\imeet(\Lam{xy}{y})^{\A}$.
  \end{enumerate}
\end{theorem}

Recall that the separator $\S\subseteq\A$ is a principal filter
when $\S={\uparrow}\{s_0\}$ for some $s_0\in\S$ or, equivalently,
when $\S$ has a smallest element~$s_0=\min(\S)$.
In this case, the Heyting algebra $H:=\A/\S$ induced by the
implicative algebra~$\A$ is complete, and the corresponding tripos is
isomorphic to the forcing tripos $\Hom_{\Set}(\Anon,\,H)$.
(Let us insist on the fact that this does not imply that the lattice
underlying~$\A$ is a complete Heyting algebra.)

However in practice, most implicative algebras come with a
separator~$\S$ that is \emph{finitely generated}, in the sense that
$\S$ is the smallest separator containing some finite subset of~$\A$.
This is typically the case:
\begin{enumerate}
\item In intuitionistic realizability, when the implicative
  algebra~$\A$ is induced by an ordered combinatory algebra (OCA)
  whose elements are generated by finitely many combinators (via
  application).
  The archetypal example of such an OCA is the combinatory algebra
  formed by all closed $\lambda$-terms (or combinatory terms) up to
  $\beta$-conversion, that is generated by the two terms
  $\mathbf{K}\equiv\Lam{xy}{x}$ and
  $\mathbf{S}\equiv\Lam{xyz}{xz(yz)}$ (via application).
\item In classical realizability, when the implicative algebra~$\A$ is
  induced by an abstract Krivine structure (AKS) whose set of
  proof-like terms is generated by finitely many instructions
  (via application).
  As a matter of fact, all the examples of AKSs presented by Krivine
  in~\cite{Kri09,Kri12,Kri16,Kri18} are of this form.
\end{enumerate}
In the particular (but frequent) case where the separator~$\S$ is
finitely generated, it is easy to see that~$\S$ is a principal filter
of~$\A$ as soon as~$\S$ is a filter, or equivalently: as soon as~$\S$
contains the \emph{non-deterministic choice operator}
$$\Fork^{\A}~:=~\bigimeet_{\mcent{a,b\in\A}}(a\to b\to a\imeet b)
~=~(\Lam{xy}{x})^{\A}\imeet(\Lam{xy}{y})^{\A}\,.$$
(Intuition: since $\Fork^{\A}a\,b\ile a$ and $\Fork^{\A}a\,b\ile b$
for all $a,b\in\A$, the truth value $\Fork^{\A}$ can be viewed as
a program taking two arguments and non-deterministically returning
any of them.)

The above discussion recalls us that in (intuitionistic or classical)
realizability, under the realistic assumption that the language of
realizers is generated by finitely many primitives (via application),
the simple fact of enriching the programming language with a
non-deterministic choice operator has actually a dramatic impact on
the corresponding realizability models, that turn out to be isomorphic
to forcing models.

\subsection{Assemblies and forcing}
Coming back to the category $\AsmA$ of assemblies on~$\A$, the
situation of forcing is captured by the following theorem:
\begin{theorem}
  For each implicative algebra~$\A$, the following assertions are
  equivalent:
  \begin{enumerate}
  \item[(1)] The category $\AsmA$ is balanced.
  \item[(2)] The category $\AsmA$ is an elementary topos
  \item[(3)] The category $\AsmA$ is a Grothendieck topos
  \item[(4)] The forgetful functor $\Gamma:\AsmA\to\Set$ is full.
  \item[(5)] The adjunction $\Gamma\dashv\Delta$ is an equivalence of
    categories between $\AsmA$ and $\Set$.
  \item[(6)] The tripos $\mathbf{P}:\Set^{\op}\to\mathbf{HA}$ induced
    by~$\A$ is (isomorphic to) a forcing tripos.
  \end{enumerate}
\end{theorem}

\begin{proof}
  $(1)\limp(6)$.\quad
  Consider the assembly $S$ defined by $\Car{S}:=\S$ and
  $\Ex{S}(s):=s$ for all $s\in\S$, as well as the morphism
  $f:S\to\Delta\S$ defined as the set-theoretic identity map
  tracked by $\top\in\S$.
  From Prop.~\ref{p:MonoEpi}, it is clear that~$f$ is both monic and
  epic in~$\AsmA$.
  If we now assume that $\AsmA$ is balanced (1), then~$f$ is an
  isomorphism too, which means that the inverse morphism
  $f^{-1}:\Delta\S\to S$ fulfills the tracking condition
  $$\bigimeet_{\mcent{s\in\S}}\bigl(\Ex{\Delta\S}(s)\to\Ex{S}(s)\bigr)
  ~{}~=~{}~\bigimeet_{\mcent{s\in\S}}(\top\to s)
  ~{}~=~{}~\top\to\bigimeet\!\!\S~{}~\in~{}~\S\,.$$
  Hence $\bigl(\bigimeet\!\!\S\bigr)\in\S$ (by modus ponens), which
  means that~$\S$ is a principal filter, or equivalently, that the
  tripos induced by~$\A$ is isomorphic to a forcing tripos~(6).
  \smallbreak\noindent
  $(6)\limp(4)$.\quad If the tripos induced by~$\A$ is isomorphic to a
  forcing tripos (6), then the separator $\S\subseteq\A$ is a
  principal filter of~$\A$, generated by its smallest element
  $s_0:=\min(\S)$.
  Given any two assemblies~$X$ and~$Y$, we observe that each
  set-theoretic map $f:\Car{X}\to\Car{Y}$ is tracked by
  the truth value $(\top\to s_0)\in\S$, hence
  $\Hom_{\AsmA}(X,Y)=\Hom_{\Set}(\Car{X},\Car{Y})
  =\Hom_{\Set}(\Gamma{X},\Gamma{Y})$.
  This proves that the forgetful functor $\Gamma:\AsmA\to\Set$
  is full~(4).
  \smallbreak\noindent
  $(4)\limp(5)$.\quad
  Regarding the adjunction $\Gamma\dashv\Delta$, we observe that:
  \begin{itemize}
  \item The counit $\epsilon_X:\Gamma(\Delta{X})\to X$ ($X\in\Set$) is
    the identity map (in~$\Set$), that is a always a natural
    isomorphism, independently from the implicative algebra~$\A$.
  \item The unit $\eta_X:X\to\Delta(\Gamma{X})$ ($X\in\AsmA$) is the
    set-theoretic identity map tracked by the truth value $\top\in\S$
    (as a morphism in~$\AsmA$).
    It is always a natural transformation, but in general not a
    natural isomorphism.
  \end{itemize}
  However, if we now assume that the forgetful functor
  $\Gamma:\AsmA\to\Set$ is full~(4), we have that
  $\Hom_{\Asm}(\Delta(\Gamma{X}),X)
  =\Hom_{\Set}(\Gamma(\Delta(\Gamma{X})),\Gamma{X})
  =\Hom_{\Set}(\Car{X},\Car{X})$ (by fullness),
  which proves that the map
  $\zeta_X:=\id_{\Car{X}}\in\Hom_{\Set}(\Car{X},\Car{X})$
  is a morphism of type $\Delta(\Gamma{X})\to X$ in~$\AsmA$.
  By construction we have $\zeta_X\circ\eta_X=\id_X$ and
  $\eta_X\circ\zeta_X=\id_{\Gamma(\Delta{X})}$.
  Therefore the unit $\eta_X:X\to\Delta(\Gamma{X})$ ($X\in\AsmA$) is
  a natural isomorphism too, which proves that the adjunction
  $\Gamma\dashv\Delta$ is an equivalence of categories between~$\AsmA$
  and~$\Set$~(5).
  \smallbreak\noindent
  The remaining implications $(5)\limp(3)\limp(2)\limp(1)$
  are obvious.
\end{proof}

\bibliographystyle{plain}
\bibliography{paper}

\end{document}